\documentclass[leqno,a4paper]{article}
\usepackage{amsmath,amsthm,amssymb}

\newcommand{\cD}{\mathcal{D}}

\newcommand{\R}{{\mathbb R}}

\newcommand{\vp}{\varphi}

\newcommand{\m}{{\mu}}

\newtheorem{definition}{Definition}[section]
\newtheorem{remark}{Remark}[section]
\newtheorem{theorem}{Theorem}[section]
\newtheorem{lemma}[theorem]{Lemma}
\newtheorem{proposition}{Proposition}[section]
\newtheorem{corollary}{Corollary}[section]
\numberwithin{equation}{section}
\def\vs1{\vspace{1ex}}

\def\O{\Omega}
\def\pa{\partial}
\def\dy{\displaystyle}

\def\be{\begin{equation}}
\def\ba{\begin{array}}
\def\ea{\end{array}}
\def\ee{\end{equation}}

\begin{document}
\title{\bf\large On the global $\,W^{2,\,q}\,$ regularity for nonlinear $N-$systems \\
of the p-Laplacian type in $n$ space variables}
\author{ H.~Beir\~ao da Veiga\ and \ F.~Crispo}

\date{}
\maketitle
\begin{abstract} We consider the Dirichlet boundary value problem for nonlinear
$N$-systems of partial differential equations with $p$-growth,
$\,1<\,p<\,2\,,$ in the $n-$dimensional case. For clearness, we
confine ourselves to a particularly representative case, the well known $p$-laplacian system.\par%
We are interested in regularity results, \emph{up to the boundary},
for the second order derivatives of the solution.  We prove
$W^{2,q}$-global regularity results, for arbitrarily large values of
$q$. In turn, the regularity achieved implies the H\"{o}lder
continuity of the gradient of the solution. It is worth noting that
we cover the singular case $\,\mu=\,0\,$. See Theorem \ref{teoremaq}
below.
\end{abstract}

\vspace{0.2cm}


 \section{Introduction}
We are concerned with the regularity problem for solutions of
nonlinear systems of partial differential equations with
$p$-structure, $p \in (1,\,2]\,$, under Dirichlet boundary
conditions. In order to emphasize the main ideas we confine
ourselves to the following representative case, where  $\mu\,\geq
\,0$ is a fixed
constant:%
\begin{equation}\label{NSC}\left\{
\begin{array}{ll}\vspace{1ex}
-\,\nabla \cdot
\,\big(\,(\,\m+|\,\nabla\,u|\,)^{p-2}\,\nabla\,u\,\big)=\,f \ \mbox{
in } \O\,,
\\%
u=\,0\ \mbox{ on } \partial \O\,.
\end{array}\right .
\end{equation}
The vector field $u=(u_1(x),\cdots, u_N(x))$, $N> 1$, is defined on
a bounded domain $\,\O \subset \R^n$, $n\geq 3$. When $\mu=0\,$, the
system \eqref{NSC} is the well-known $p$-Laplacian system. \par%
It is worth noting that our interests concern global (up to the
boundary), full regularity for the second derivatives of the
solutions. Our results also hold  in the singular case $\mu=0$. For
any bounded and sufficiently smooth domain $\O$, we prove
$W^{2,q}(\O)$ regularity, for any $q\geq 2$. Therefore, we get, as a
by product, the $\alpha$- H\"{o}lder continuity, up to the boundary,
of the gradient of the solution, for any $\alpha <\, 1\,$. The
results are obtained for $\,p\,$ belonging to intervals $[C,2)$,
where $C$ are suitable constants, whose expression may be explicitly
calculated. In particular, if $\,\O\,$ is convex, solutions belong
to $W^{2,2}(\O)$ for any
$\,1\,<p\,\leq\, 2\,.$\par%
As usual, weak solutions are defined as follows (for notation and
more precise statements see the sequel).
\begin{definition}\label{noit}
Assume that $f \in W^{-1\,,p'}(\Omega)$. We say that $u$ is a
{\rm{weak solution}} of problem \eqref{NSC}
 if $\,u\,\in\, W^{1,p}_0(\O)$ satisfies
\begin{equation}\label{buf2}
\int_{\O}\ \left(\,\m+|\,\nabla\,u|\,\right)^{p-2}\,\nabla\,u\,
\cdot \nabla\, \vp \,dx\ =\,\int_{\O}\, f\, \cdot\, \vp \,dx\,,
\end{equation}
for all $\vp \in \,W^{1,p}_0(\O)$.
\end{definition}
It is immediate to verify that, if $\,\mu >\,0\,,$ sufficiently
regular weak solutions to the problem \eqref{NSC} satisfy%
\be\label{munotz}-\Delta u-(p-2) \frac{\nabla u\cdot\nabla\nabla
u\cdot \nabla u}{(\mu+\,|\nabla
u|)\,|\nabla u|}= f\left(\mu+|\nabla\,u|\right)^{2-p}\,.\ee%
Here, and in the following, we use the notation $\nabla
u\cdot\nabla\nabla u\cdot \nabla u$ to indicate the vector whose
$i^{\rm th}$ component is $\nabla u\cdot \left( \pa_j\,\nabla
u\right) \pa_j\,u_i=(\pa_l\,u_k)
\,(\pa^2_{j\,l}\,\,u_k)\,(\pa_j\,u_i)\,.$\par%

\vspace{0.2cm}

In the sequel we start by proving the existence of a (unique) strong
solution
$$
u \in \,  W^{1,p}_0(\O)\cap \, W^{2,q}(\O)\,,%
$$
of problem \eqref{munotz}, under homogeneous Dirichlet boundary
conditions, in the case $\mu>\,0\,$ (see Theorem
\ref{teoremaqreg1}). Clearly, $u$ solves \eqref{buf2}. Furthermore,
we prove that the $\, W^{2,q}(\O)$ norms of the above strong
solutions are uniformly bounded with respect to $\mu\,$. This allows
us, by passing to the limit as $\mu \rightarrow \,0\,,$ to extend
the $W^{2,q}(\O)$ regularity result to weak solutions of problem
\eqref{NSC} in the case $\mu=\,0\,$ (see Theorem \ref{teoremaq}).

\vspace{0.2cm}

The regularity issue for systems like \eqref{NSC} has received
substantial attention, mostly concerned with the scalar case
($N=1$), and with $C^{1,\alpha}_{loc}$-regularity. Here and in the
following, by local regularity we mean interior regularity. The
pioneering result dates back to Ural'tseva \cite{ura}, where, for
$p>2$ and $N=1$, the author proves $C^{1,\beta}_{loc}$-regularity
for a suitable exponent $\beta$. Still in the case $\,N=\,1\,$ we
recall the following contributions. In \cite{Tolk2} the author
proves $W^{2,p}_{loc}$-regularity for any $p<2$, and also
$W^{2,2}_{loc}$-regularity, for $p>2$. In \cite{Lieb88}, for
$p>2\,,$ the author proves $C^{1,\beta}$-regularity up to the
boundary, in $\Omega\subset \R^n\,$. In \cite{Liu} the author shows,
for any $p\,\in (1,2)\,,$ $W^{2,2}\cap C^{1,\alpha}$-regularity up
to the boundary, in
$\Omega\subset \R^2$.\par%
For systems (solutions are $N$-dimensional vector fields, $N>1$), we
recall \cite{acerbi} for $p\in (1,2)$, \cite{GM} and \cite{uhlenb}
for $p>2$, and \cite{Hamburger} for any $p>1$. The results proved in
papers \cite{acerbi}, \cite{GM} and \cite{uhlenb} are local.
Moreover all these papers deal only with homogeneous systems and the
techniques, sometimes quite involved, seem not to be directly
applicable to the non-homogeneous setting. In particular,
\cite{acerbi} is the only paper in which the $L^2_{loc}$-regularity
of second derivatives is considered. The results below are, in the
non-scalar case, the first regularity results up to the
boundary, for the second derivatives of solutions. \par%
 For related results and for an extensive bibliography we also
refer to papers \cite{AM}, \cite{BDVCRIplap}, \cite{DB}, \cite{DBM},
\cite{FuM}, \cite{FS},  \cite{Lieb}, \cite{MP}, \cite{Ming} and
references therein.
\par%
We observe that we do not consider a more general dependence on
$\nabla u$, as for instance $\varphi(|\nabla u|)\, \nabla u$, under
 suitable assumptions on the scalar function $\varphi$, just to
emphasize the core aspects of the results and to avoid additional
technicalities. For the same
 reason we avoid the introduction of lower order terms. Note that another, very
similar, representative case can be obtained with the regular term
$(\,\m+|\,\nabla\,u|^2\,)^{\frac{p-2}{2}}\,$ in place of
$(\,\m+|\,\nabla\,u|\,)^{p-2}$ in \eqref{NSC}. This latter function
is only Lipschitz continuous, hence in this case it seems not
possible to get stronger regularity results.  Finally one could also
extend the results to non-homogeneous Dirichlet boundary
conditions, if the boundary data belongs to a suitable $W^{2,q}$-space.\par%

\vspace{0.2cm}

\begin{remark}
\rm{Different, more intricate, proofs of Theorem \ref{teoremaq} and
its corollaries were given in \cite{BDVCRIplap}, in the particular
case $\,N=\,n=\,3\,.$ In \cite{BDVCRIplap} we also consider the case
where $\,\nabla\, u\,$ is replaced by $\cD\, u\,=\,\frac
12\,(\,\nabla\, u\,+\,\nabla \,u^T\,)\,,$ and $\,p
\in\,(1,\,+\infty)\,.$}
\end{remark}
\section{Notation and statement of the main results}
Throughout this paper we denote by $\Omega$ a bounded
$n$-dimensional domain, $n\geq 3$, with smooth boundary, which we
assume of class $C^2$, and we consider the usual homogeneous
Dirichlet boundary conditions \be\label{diri}
u_{|\partial\Omega}=0.\ee
\par By
$L^p(\O)$ and $W^{m,p}(\O)$, $m$ nonnegative integer and
$p\in(1,+\infty)$, we denote the usual Lebesgue and Sobolev spaces,
with the standard norms $\|\cdot\|_{L^p(\O)}$ and
$\|\,\cdot\,\|_{W^{m,p}(\O)}$, respectively. We usually denote the
above norms by $\|\cdot\|_{p}$ and $\|\,\cdot\,\|_{m,p}$, when the
domain is clear. Further, we set $\|\cdot\|=\|\cdot\|_{2}$. We
denote by $W^{1,p}_0(\O)$ the closure in $W^{1,p}(\O)$ of
$C^\infty_0(\O)$ and by $W^{-1,p'}(\O)$, $p'=\,p/(p-1)$, the strong
dual of $W^{1,p}_0(\O)$ with norm $\|\,\cdot\,\|_{-1,p'}$.
\par%
In notation concerning  norms and functional spaces, we do not
distinguish between scalar and vector fields. For instance
$L^p(\Omega;\R^N)= [L^p(\O)]^N$, $N>1$, is simply $L^p(\O)$.
\par
We use the summation convention on repeated indexes. For any given
pair of matrices $B$ and $C$ in $R^{Nn}$ (linear space of $N\times
n$-matrices), we write $B\cdot C\equiv B_{ij}\,C_{ij}$.
\par
We denote by the symbols $c$, $c_1$, $c_2$, etc., positive constants
that may depend on $\mu$; by capital letters, $C$, $C_1$, $C_2$,
etc., we denote positive constants independent of $\mu
\geq\,0\,$(eventually, $\,\mu\,$ bounded from above). The same
symbol $c$ or $C$ may denote different constants, even in the same
equation.

\par%
 We
set $\partial_i \,u=\,\frac{\pa\, u}{\pa\, x_i}$,
$\,\partial_{ij}^2\,u=\,\frac{\pa^2\, u}{\pa\, x_i\pa\,x_j}$.
Moreover we set $(\nabla u)_{ij}=\pa_j\,u_i$. We denote by $D^2u$
the set of all the second partial derivatives of $u$. Moreover we
set \be\label{SDstar} |\,D^2u\,|^2:=\sum_ {i=1}^N\sum_{j,h=1}
^n\!\!\left|\,\pa_{jh}^2\,u_i\,\right|^2\,.\ee%
\par Let us introduce the definition of weak solution of
 problem \eqref{NSC}--\eqref{diri}.
\begin{definition}\label{noitaa}
Assume that $f \in W^{-1,p'}(\O)$. We say that $u$ is a {\rm{weak
solution}} of problem \eqref{NSC}--\eqref{diri},
 if $\,u\,\in\, W^{1,\,p}(\O)$ satisfies
\begin{equation}\label{buf2aa}
\int_{\O}\ \left(\,\m+|\,\nabla\,u|\,\right)^{p-2}\,\nabla\,u\,\cdot
\nabla\, \vp \,dx\ =\,\int_{\O}\, f\, \cdot\,  \vp \,dx\,,
\end{equation}
for all $\vp \in \,W_0^{1,p}(\O)$.
\end{definition}
\par We recall that the existence and uniqueness of a weak solution
can be obtained by appealing to the theory of monotone operators,
following J.-L. Lions \cite{lions}.
\par Before stating our main results, let us recall two well known inequalities for the Laplace
operator.
The first, namely%
\be\label{lad2} \|\,D^2\,v\,\|\leq C_1\,\|\,\Delta\, v\,\|\,,\ee
holds for any function $v\in W^{2,2}(\O)\cap W_0^{1,2}(\O)\,$, with
$C_1=C_1(\Omega)\,$. Note that $C_1=1$ if $\O$ is convex. For
details we refer to \cite{Lad} (Chapter I, estimate (20)). The
second kind of estimates which we are going to use says that
\be\label{ladaq}\|D^2\,v\|_{q}\leq \,C_2\|\Delta v\|_q\,,\ee%
for $v\in W^{2,q}(\O)\cap W_0^{1,q}(\O)$, $q\geq 2$, where the
constant $C_2$ depends on $q$ and $\O$. It relies on standard
estimates for solution of the Dirichlet problem for the Poisson
equation.  Actually, there are two constants $K_1$ and
$K_2$, independent of $q$, such that%
\be\label{yud} K_1\, q\leq C_2\leq K_2\, q\,.\ee%
 Similarly, one has
\be\label{lad1q}\|\,\,v\,\|_{\,2,q}\leq \,C\|\Delta v\|_q\,,\ee
where the constant $C$ depends on $q$ and $\O$. For further details
we refer to \cite{kos} and \cite{yud}.
\par%
We set \be\label{rq}r(q)=\left\{\begin{array}{ll}\dy
 \frac{nq}{n(p-1)+q(2-p)}
&\dy \mbox{ if }\ q\in [2,\,n]\,,\\
\hskip1cm q & \dy \mbox{ if }\ q \geq\, n\,.
 \end{array}\right .\ee
Note that $r(q)>q$ for any $q<n$. Clearly, in \eqref{rq}, $r(q)=n$
in booth cases. Our main results is the following.
\begin{theorem}\label{teoremaq}
Let $p\in (1,2]$ satisfy $\,(2-p)\,C_2<\,1\,$, where $C_2$ is given
by \eqref{ladaq}. Assume that $\mu\geq 0$. Let $f\in L^{r(q)}(\O)$
for some $q\geq\,2\,,$ $\,q\not= n\,,$ and let $u$ be the unique
weak solution of problem \eqref{NSC}. Then $u$ belongs to
$W^{2,q}(\O)$. Moreover, the following estimate holds
\begin{equation}\label{dnq}
 \|u\|_{2,q}\leq C
 \,\left(\|f\|_q+\|f\|_{r(q)}^\frac{1}{p-1}\right)\,.
\end{equation}
\end{theorem}

\begin{corollary}\label{corollaryq2}
Let  $p$, $\mu$ and $f$ be as in Theorem \ref{teoremaq}. Then, if
$\,q>n$, the weak solution of problem \eqref{NSC} belongs to $
C^{1,\alpha}(\overline\O)$, for $\alpha= 1-\frac nq$.
\end{corollary}

In particular, when $q=2$, one has the following corollary.

\begin{corollary}\label{teorema}
Let $p\in (1,2]$ satisfy $(2-p)\,C_1<\,1\,$, where $C_1$ is given by
\eqref{lad2}. Assume that $\mu\geq 0$. Let $f \in L^{r(2)}(\Omega)$
and let $u$ be the unique weak solution $u$ of problem \eqref{NSC}.
Then $u$ belongs to $W^{2,2}(\O)$. Moreover, there is a constant $C$
such that
\begin{equation}\label{dn}
 \|\,u\,\|_{2,2}\leq C\left(\|f\|+\|f\|_{r(2)}^\frac{1}{p-1}\right)\,.
 \end{equation}
If $\,\O\,$ is convex the result holds for any $\,1\,<p\,\leq\,
2\,.$
\end{corollary}
It is worth noting that in the limit case $p=2$, when system
\eqref{NSC} reduces to the Poisson equations, we recover exactly the
well known result
$$ \|\,u\,\|_{2,q}\leq C\,\|f\|_q\,,$$
since $\,r(q)=\,q\,$ for $\,p=\,2\,$.\par%
Note that in estimates \eqref{dnq} and \eqref{dn}, the terms $\|f\|$
and $\|f\|_q$ can be replaced by $1$.
\begin{remark} \rm{One could also consider the case where $f\in L^n(\Omega)$. We omit this
further case and leave it to the interested reader. In this regard
we stress that our interest mostly concerns the maximal
integrability of the second derivatives of the solution.}
\end{remark}

\section{Proof of Theorem \ref{teoremaq}. The case $\mu>0\,.$}
In this section we assume that $\mu>0\,.$ Let us consider the
following system \be\label{reg1}-\Delta u-(p-2) \frac{\nabla
u\cdot\nabla\nabla u\cdot \nabla u}{(\mu+\,|\nabla u|)\,|\nabla u|}=
f\left(\mu+|\nabla\,u|\right)^{2-p}\,,\ee where we have used the
notation $\nabla u\cdot\nabla\nabla u\cdot \nabla u$ to denote the
vector whose $i^{\rm th}$ component is $\nabla u\cdot \left(
\pa_j\,\nabla u\right) \pa_j\,u_i=(\pa_l\,u_k)
\,(\pa^2_{j\,l}\,\,u_k)\,(\pa_j\,u_i)\,.$ Formally this system can
be obtained from system \eqref{NSC} by computing the divergence on
the left-hand side and then multiplying the equation by
$\left(\mu+|\nabla\,u|\right)^{2-p}$.\par%
It is immediate to verify that if $u$ is a sufficiently regular
solution of \eqref{reg1}, say $\,u \in\,W^{2,2}(\O)$, then $u$ is a
weak solutions of \eqref{NSC}. So, from the uniqueness of weak
solutions of \eqref{NSC}, it follows that to prove  Theorem
\ref{teoremaq} under the assumption $\,\mu>0\,$ it is sufficient to
prove the following result for strong solutions.
\begin{theorem}\label{teoremaqreg1}
Let $p\in (1,2]$ satisfy $\,(2-p)\,C_2<\,1\,$, where $C_2$ is given
by \eqref{ladaq}. Assume that $\mu> 0$. Let $f\in L^{r(q)}(\O)$ for
some $q>2$ and $q\not= n\,.$ Then, there is a strong solution $u\,
\in\, W^{2,q}(\O)$ of problem \eqref{reg1}--\eqref{diri}. Moreover,
the following estimate holds
\begin{equation}\label{dnqpf}
 \|u\|_{2,q}\leq C
 \,\left(\|f\|_q+\|f\|_{r(q)}^\frac{1}{p-1}\right)\,.
\end{equation}
\end{theorem}
\par In the sequel we appeal to the following fixed point theorem in
order to prove Theorem \ref{teoremaqreg1}.

\begin{theorem}\label{fixo}
Let $X$ be a reflexive Banach space and $K$ a non-empty, convex,
bounded, closed subset of $X$. Let $F$ be a map defined in $K$, such
that $F(K) \subset \,K$.\par%
Assume that there is a Banach space $Y$ such that:\par%
i) $X \subset \,Y\,,$ with compact (completely continuous) immersion.\par%
ii) If $v_n \in \,K$ converges weakly in $X$ to some $v \in K$ then
there is a subsequence  $v_m$ such that $F(v_m) \rightarrow
\,F(v)\,$ in $Y$.\par%
Under the above hypotheses the map $F$ has a fixed point in $K$.
\end{theorem}

For the proof and some comments see section \ref{pointf}.\par%
In the sequel we appeal to the above theorem with $X=\,W^{2,\,q}\,$
and $\,Y=\,L^{q}$. Clearly, point i) in Theorem \ref{fixo} holds.

\vspace{0.2cm}

\begin{proof}[Proof of Theorem \ref{teoremaqreg1}]
For any $v\,\in\, W^{2,q}\cap W_0^{1,q}\,$ we define
$\,C_3=\,C_3(q)$ by \be\label{c7}\ba{ll}\vspace{1ex}\displaystyle \|
\nabla\,\,v\|_{q^*}\leq \,C_3\,\| \Delta\,v\|_q\,, \mbox { if } q\in
(2,\,n),\\
\dy \| \nabla\,\,v\|_{\infty}\leq  \,C_3\,\| \Delta\,v\|_q\,, \mbox
{ if } q\in (n,+\infty). \ea\ee These estimates can be easily
obtained by applying the Sobolev embeddings and then using estimate
\eqref{lad1q}. \par %
Define $\,\delta\,$ by
$$\delta=\,1-\,(2-\,p)\,C_2\,,$$
where $C_2$ is given by \eqref{ladaq}, and fix a positive real $a$
by
$$
1+\,2\,C_3^{2-\,p} \,a^{2-\,p}\leq \,a\,\delta\,.
$$
Note that, under our assumptions, $\delta>0$. It is worth noting
that $\,\delta\,$, $\,a\,$ and $\,b\,$ are
constants of type $C$\,.\par%
 Define
$$
K=\{ v\in\,W^{2,\,q}(\O)\,:\,\| \Delta\,v\|_q \leq\,R\,,\, v=0\
\mbox{ on } \pa\O\}\,,
$$
where
$$
R=\,a\,\left(\|f\|_q+\,\|f\|_{r(q)}^\frac{1}{p-1}\,\right)\,.
$$
Let $f\in L^{r(q)}$ be given. For each  $v\in\,K$ define $u=\,F(v)$ as being the solution to the linear problem%
\be\label{rega1}\left\{\ba{ll}%
\dy -\Delta u=(p-2) \frac{\nabla v\cdot\nabla\nabla v\cdot \nabla
v}{(\mu+\,|\nabla v|)\,|\nabla
v|}+f\left(\mu+|\nabla\,v|\right)^{2-p}\,,\ \mbox{ in } \O\,,\\
u=0\,, \ \mbox{ on }\pa\O\,.
\ea\right .\ee%
To apply Theorem \ref{fixo} we start by showing that $F(K)
\subset\,K\,$. Note that if the right-hand side of \eqref{rega1}
belongs to $L^q(\O)$, from well known results on the Poisson
equation, there exists a unique $u\in W^{2,q}(\O)$ solving the
Dirichlet problem \eqref{rega1}.
For $q\in (2,\,n)$, by using \eqref{c7}$_1$ we have%
$$\|\,|\nabla v|^{2-p}\,f\|_q\leq \dy\|\,\nabla
v\,\|_{q^*}^{2-p}\, \|\,f\|_{r(q)} \, \leq\,C_3^{2-\,p}\,\|\,\Delta
v\,\|^{2-\,p}_q\,\|\,f\|_{r(q)}\,.$$
For $q>n$, by using \eqref{c7}$_2\,,$ and by recalling that
$r(q)=q\,$ if $q>\,n\,,$ we have
$$
\|\,|\nabla v|^{2-p}\,f\|_q\leq \dy\|\,\nabla v\,\|_{\infty}^{2-p}\,
\|\,f\|_{q} \, \leq\,C_3^{2-\,p}\,\|\,\Delta
v\,\|^{2-\,p}_q\,\|\,f\|_{q}\,.$$
So, in both the cases, \be\label{cmha2}\|\,|\nabla
v|^{2-p}\,f\|_q\leq \,C_3^{2-\,p}\,\|\,\Delta
v\,\|^{2-\,p}_q\,\|\,f\|_{r(q)}\,.\ee%
Therefore, since the first term on the right-hand side of
\eqref{rega1} obviously belongs to $L^q(\O)$, there exists a unique
$u\in W^{2,q}(\O)$ solving the Dirichlet problem \eqref{rega1}.
\par It remains to show that $u$ satisfies the estimate $\|\Delta\, u\|_q\leq
R\,$. We multiply both sides of equation \eqref{rega1} by $-\,\Delta
u\,|\Delta u|^{q-2}$, and integrate in $\Omega$. We get (for details
see the appendix)
$$\int_{\O}|\,\Delta u\,|^q\,dx\\ \dy \hfill\leq (2-p)\int_{\O}|\,D^2
v\,|\,|\,\Delta u\,|^{q-1}\,dx+\int_{\O}(\,\m+\,|\nabla
v|\,)^{2-p}\,|\,f\,|\,|\,\Delta u\,|^{q-1}\,dx\,. $$%
The H\"{o}lder's inequality and the inequality $(\m+\,|\nabla
v|)^{2-p}\leq 1+\,|\nabla v|^{2-p}$ yield
\be\label{lap1w}\ba{ll}\vs1\dy \|\,\Delta u\,\|_q^q\,\leq
&\dy\!\!(2-p)\,\|\,D^2v\,\|_q \|\,\Delta u\,\|_q^{q-1}\\\hfill&\dy
+\,\|\,f\,\|_q\|\,\Delta u\,\|_q^{q-1}+ \|\,|\nabla
v|^{2-p}\,f\,\|_q\|\,\Delta u\,\|_q^{q-1}\,,\ea \ee and, by dividing
both sides by $\|\,\Delta u\,\|_q^{q-1}$,  one has%
\be\label{lap122we} \|\,\Delta u\,\|_q\leq (2-p)\,\|\,D^2\,v\,\|_q
+\,\|\,f\,\|_q+\|\,|\nabla v|^{2-p}\,f\,\|_q. \ee%
Let us estimate the last term on the right-hand side of
\eqref{lap122we}. Since $v\in K$, one has \be\label{del2p}
 \| \Delta\,v\|^{2-\,p}_q
\leq\,a^{2-\,p}\,\left(\|f\|^{2-\,p}_q+\,\|f\|_{r(q)}^\frac{2-\,p}{p-1}\right)\,.
\ee Hence, from \eqref{cmha2}, by using \eqref{del2p} and
$$
\|f\|^{2-\,p}_q \,\|f\|_{r(q)} \leq\,\|f\|_q
+\,\|f\|^{\frac{1}{p-\,1}}_{r(q)}\,,
$$
one gets $$\|\,|\nabla v|^{2-p}\,f\|_q\leq
C_3^{2-\,p}\,a^{2-\,p}\,\left(\,\|f\|_q +
2\,\|f\|^{\frac{1}{p-\,1}}_{r(q)}\right)\,.
$$
Therefore \eqref{lap122we} becomes \be\label{esdel}\|\,\Delta
u\,\|_q\leq \left( (2-p)\,C_2\,\,a\,
+1+\,2\,C_3^{2-\,p}\,a^{2-p}\right)\left(\|f\|_q+\,\|f\|_{r(q)}^\frac{1}{p-1}\,\right)\,,\ee
where we have appealed to \eqref{ladaq}. Finally from the definition
of $\delta$, it readily follows that $\,u \in\,K\,$. So, $F(K)
\subset\,K\,.$

\vspace{0.2cm}

To end the proof of Theorem \ref{teoremaqreg1} it is sufficient to
show the following result (which corresponds to point ii) in Theorem
\ref{fixo}).
\begin{proposition}\label{opo}
Let $v_n \rightharpoonup\,v\,$ weakly in $W^{2,\,q}\,,$ where $v_n
\in \,K\,$. If $\,u_n=\,F(v_n)$ are the solutions to the problem
\begin{equation}\label{reg123n}%
-\Delta u_n=(2-p) \frac{\nabla v_n \cdot\nabla\nabla v_n\cdot \nabla
v_n}{(\mu+\,|\nabla v_n|)\,|\nabla v_n|}+
f\left(\mu+|\nabla\,v_n|\right)^{2-p}\,,%
\end{equation}
then there is a subsequence $v_m$ of $v_n$ such that $u_m=\,F(v_m)
\rightarrow \,u$ in $Y=\,L^q\,,$ where $u=\,F(v)$.%
\end{proposition}

In the sequel we use the label \eqref{reg123n}$_l$ to mean that the
sequences $u_n$ and $v_n$ are replaced by subsequences $u_l$ and
$v_l $ respectively. For instance we can denote identity
\eqref{reg123n} also by \eqref{reg123n}$_n$.\par%
Since $u_n\in K$, there is a subsequence $u_k$ and an element $u
\in\,K$ such that $u_k \rightharpoonup \,u$ weakly in $W^{2,\,q}$
(since this space is reflexive). In particular $\,-\Delta\,u_k
\rightharpoonup \,-\Delta\, u$ weakly in $L^q\,.$ Moreover, $u_k
\rightarrow \,u$ strongly in $W^{s,\,q}$, for each $s<\,2\,,$ hence
$u_k \rightarrow \,u$ in $Y=\,L^q\,$.\par%
The proof is accomplished by showing that one can pass to the limit
in \eqref{reg123n}$_m$, along subsequences $u_m$ and $v_m$, to
obtain
\be\label{reg123}%
-\Delta u=(p-2) \frac{\nabla v \cdot\nabla\nabla v\cdot \nabla
v}{(\mu+\,|\nabla v|)\,|\nabla v|}+
f\left(\mu+|\nabla\,v|\right)^{2-p}\,.%
\ee%
To prove the proposition it is sufficient to consider the equation
\eqref{reg123n}$_k$ and to show that there is a subsequence $v_m$ of
$v_k$ such that each of the two terms in the right hand side of
\eqref{reg123n}$_m$ converge, in the distributional sense, to the
corresponding terms in equation \eqref{reg123}. This verification
would be quite immediate. However, we rather prefer to prove the
convergence in a topology stronger than the distributional one.

\vspace{0.2cm}

For convenience we set
$$
A(w)=:\, \frac{\partial_1 w \, \partial_2 w}{(\mu+\,|\nabla
w|)\,|\nabla w|}\,,
$$
where $\,\partial_1 w\,$ and $ \,\partial_2 w\,$ denotes any couple
of arbitrary, fixed, partial derivatives.
\begin{lemma}\label{limum}
There is a subsequence $v_m$ of $v_k$ such that
$$
A(v_m)  \rightarrow \,A(v)\,,
$$
strongly in $L^t$, for each  $t>\,1$.%
\end{lemma}
\begin{proof}
Since, in particular, $v_k \rightarrow \,v$ in $\,W^{1,q}\,$, it
follows, by a classical result, that almost everywhere convergence
of the gradient in $\Omega$ also holds, for some $v_m$. So,
$\,A(v_m) \rightarrow \,A(v)\,$, a.e. in $\Omega$. Further,
$\,|\,A(v_m(x)\,)\,|^t\, \leq\,1\,,$ point-wisely. It follows, from
the reflexivity of $L^t$, that $\,A(v_m)$ is weakly convergent in
$L^t\,.$ Due to the a.e. convergence, see \cite{lions}, chap. I,
Lemma 1.3, the weak limit is just $A(v)$. So,
$$
A(v_m)  \rightarrow \,A(v)\,,
$$
weakly in $L^t$, for each finite $t$. This last property, together
with $\,\|\,A(v_m)\,\|_t  \rightarrow \,\|\,A(v)\|_t\,$ implies
strong convergence, thanks to a classical theorem, see \cite{riesz}
(Chap.2, n. 37). The above norm-convergence follows by appealing to
Lebesgue's dominated convergence theorem.
\end{proof}
Next, we prove that each of the two terms in the right hand side of
\eqref{reg123n}$_m$ converge to the corresponding terms in equation
\eqref{reg123}. We start by the first term. Each single addend has
the form $\,A(v_m)\,\partial^2\,v_m\,$, where $\,\partial^2\,w\,$
denotes an arbitrary, fixed, second order derivative. We prove the
following result.
\begin{lemma}\label{marras}
One has
$$
A(v_m)\,\partial^2\,v_m \rightharpoonup \,A(v)\,\partial^2\,v
$$
weakly in  $L^s$, for each $\,s<\,q\,$.%
\end{lemma}
\begin{proof}
Set $\,\,g=\,A(v)\,,$ $\,g_m=\,A(v_m)\,,$ $\,h=\,\partial^2\,v\,,$
and  $\,h_m=\,\partial^2\,v_m\,.$ Clearly, $\,h_m \rightharpoonup
\,h\,$ weakly in $\,L^q\,.$ Moreover, by the previous lemma, $\,g_m
\rightarrow\,g\,$ strongly in $\,L^t\,$,
$\,t=:\,\frac{q\,s}{q-\,s}\,$. Moreover,  and $\,h_m \rightharpoonup
\,h\,$ weakly in $\,L^q\,.$\par%
Write%
\begin{equation}\label{milos}%
g_m\,h_m -\,g\,h=\,g\,(h_m-\,h)+\,(g_m-\,g)\,h_m\,,%
\end{equation}
and let $\,\phi \in\,L^{q'}\,$. Since $\,g(x)\,$ is bounded it
follows that  $\,g\,\phi \in\,L^{q'}\,.$ So the quantity
$$
<\,g\,(h_m-\,h),\,\phi\,> =\,<\,(h_m-\,h),\,g\,\phi\,>
$$
goes to zero as $\,m \rightarrow\,\infty\,$. This proves the weak
convergence to zero, in $\,L^q\,,$ of the first term in the right hand side of \eqref{milos}.\par%
On the other hand, by H\"{o}lder's inequality,
$$
\|\,(g_m-\,g)\,h\,\|^s_s
\leq\,\|\,g_m-\,g\,\|^s_{\frac{q\,s}{q-\,s}}\,\|\,h\,\|^s_q\,.
$$
This proves the strong convergence to zero, in $\,L^s\,,$ of the
second term in the right hand side of \eqref{milos}. In conclusion,
the first term in the right hand side of \eqref{reg123n}$_m$
converges to the first term in the right hand side of
\eqref{reg123}.\end{proof}%
Finally, the convergence of the second term in the right hand side
of \eqref{reg123n}$_m$ to the corresponding term in \eqref{reg123}
holds, since
$$
|\,(\mu+|\nabla\,v_m|\,)^{2-p} -\,(\mu+|\nabla\,v|\,)^{2-p}\,|
\leq\,\frac{2-\,p}{\mu^{p-\,1}}\,|\nabla\,v_m -\,\nabla\,v|\,.
$$
By Cauchy-Schwartz inequality
$$
\|\,f\,(\mu+|\nabla\,v_m|\,)^{2-p}
-\,f\,(\mu+|\nabla\,v|\,)^{2-p}\,\|_{  \frac{q}{2}}
\leq\,\frac{2-\,p}{\mu^{p-\,1}}\,\|\,f\,\|_q\, \|\,\nabla\,v_m
-\,\nabla\,v\,\|_q \,,
$$
and the right-hand side goes to zero thanks to the compact embedding
of $\,W^{2,\,q}\,$ in $\,W^{1,\,q}\,$.\par%
The solution $u$ obviously satisfies \eqref{dnqpf}, as $u\,\in\,
K\,$.
\end{proof}
\section{Proof of Theorem \ref{teoremaq}. The case $\mu =\,0\,.$}
In the previous step we have obtained estimates on the $L^q$-norm of
the second derivatives,  uniformly in $\mu\,$, $\,\mu >\,0\,.$ Let
us denote by $u_\mu$ the sequence of solutions of \eqref{NSC} for
the different values of $\mu>0$. We have shown that the sequence
$(u_\mu)$ is uniformly bounded in $W^{2,q}(\O)$. Therefore,
 there exists a vector field $u\in W^{2,q}(\O)$ and a
subsequence, which we continue to denote by $(u_\mu)$, such that
$(u_\mu)\rightharpoonup u$ weakly in $W^{2,q}(\O)$, and, by
Rellich's theorem, strongly in $W^{1,s}(\O)$, for any $\,s$ if
$\,q>\,n\,,$ and for $\,s<\,q^*\,$ if $\,q<\,n$. In particular
$(u_\mu)$ converges to $u$ strongly in $W^{1,p}(\O)$. Let us prove
that%
\be\label{a2}%
\int_{\Omega} \,\left(\,\mu+|\nabla u\,| \,\right)^{p-2}\,\,\nabla
u\, \cdot \nabla \varphi \,dx= \lim_{\mu\to 0^+}\left\{\int_{\Omega}
\,\left(\,\mu+|\nabla u_\mu|\,\right)^{p-2}\,\nabla u_\mu\, \cdot
\nabla\varphi \,dx\right\}\,,\ee for any $\varphi
\in\,W^{1,p}_0(\O)\,$. We recall the following well known estimate
(see, for instance, \cite{DER})
\be\label{tensorS1}|\,(\mu+|A|)^{p-2}A-(\mu+|B|)^{p-2}B| \leq\,
C\,\frac{|A-B|}{(\mu+|A|+|B|)}{\!\atop ^{{2-p}}}\,,\ee for any pair
$A$ and $B$ in $\R^{Nn}$, where $C$ is a positive constant
independent of $\mu$.\par%
By applying \eqref{tensorS1} and then H\"{o}lder's inequality, we
get%
\be\label{aaaa}\ba{ll}\vs1\dy \left|\int_{\Omega}
\,\left(\,\mu+|\nabla u\,| \,\right)^{p-2}\,\,\nabla u\, \cdot
\nabla \varphi \,dx -\,\int_{\Omega} \,\left(\,\mu+|\nabla
u_\mu|\,\right)^{p-2}\,\nabla u_\mu\, \cdot \nabla\varphi
\,dx\,\right|\\
\vs1\dy\leq C\,\int_{\Omega}\! \,\left(\,\mu+|\nabla u\,|+|\nabla
u_\mu|\,\right)^{p-2}\,|\,\nabla u-\nabla u_\mu\,|\,|\nabla
 \varphi|\, dx\,\\\vs1\dy
  \leq C\,\int_{\Omega}\! \,\left|\,\nabla u-\nabla
u_\mu\,\right|^{p-1}\,|\nabla
 \varphi|\, dx\,\leq C\,\|\,\nabla
u_\mu\!-\nabla u\,\|_p^{p-1}\,\|\,\nabla \varphi\,\|_p\,.\ea\ee The
right-hand side of the last inequality tends to zero, as $\mu$ goes
to zero, thanks to the strong convergence of $u_\mu$ to $u$ in
$W^{1,p}(\O)$. This proves \eqref{a2}. Finally, for each $\mu>0\,,$
the right-hand side of \eqref{a2} is equal to $\int_\O f\cdot
\vp\,dx\,.$ So, $u$  satisfies the integral identity \eqref{buf2aa}.
Hence $u$ is a weak solution of \eqref{NSC}, and belongs to
$\,W^{2,q}(\O)$. Finally, \eqref{dn} follows since
$\|\,\,u\,\|_{2,q}\leq\dy \liminf_{\mu\to 0^+} \|\,u_\mu\,\|_{2,q}$.

\vspace{0.2cm}

The Corollary \ref{corollaryq2} is an immediate consequence of
Theorem \ref{teoremaq}, by using the regularity of the domain and
the Sobolev embedding. \vskip0.2cm The results in Corollary
\ref{teorema} can be obtained by replacing in the proof of Theorem
\ref{teoremaq}, hence in the proof of Theorem \ref{teoremaqreg1},
the constant $C_2$ with the constant $C_1$. The last assertion in
Corollary  \ref{teorema} follows from the validity of \eqref{lad2},
for a smooth convex domain, with $C_1=1$. We omit further details.
\section{The fixed point theorem. Proof and remarks.}\label{pointf}

Theorem \ref{fixo} is a simplification of an idea introduced in
reference \cite{BVdiv} to prove existence of strong solutions to
initial boundary value problems for non-linear systems of evolution
equations, specially in Sobolev spaces. See the section 3, in the
above reference. Successively, the method has been applied with
success to many other problems, in particular to the compressible
Euler equations (see \cite{BV2}). Main requirements, in
applications, are the reflexivity of the Banach space $X$, and its
sufficiently strong topology. Shauder's fixed point theorem is
applied with respect to a quite arbitrary ``container space'' $Y$.
Roughly speaking, the above two properties allow us to trivialize
both compactness and continuity requirements, respectively. So, to
apply the theorem, the main point is to show that $F(K) \subset
\,K\,,$ for some convex, bounded, closed subset $K$.
\begin{proof}[Proof of Theorem \ref{fixo}]
Obviously $K$ is convex, bounded, and pre-compact in $Y$. \par Let
$y_n \in \,K$ converge to some $y$ in the $Y$ norm. We start by
showing that $K$ is closed, hence compact, in $Y$, and that the
sequence $\,y_n\,\rightharpoonup\, y$ weakly in $\,X\,$.
  Since $K$ is $X$-bounded, and $X$
is reflexive, there is a subsequence $y_m\,$ which is $\,X$-weakly
convergent to some $u \in \,X\,$. Since the immersion $ X \subset
\,Y$ is continuous, $y_m\,$ is also weakly convergent to $u$ in $Y$.
Since, by assumption, this sequence is strongly convergent in $Y$ to
$y$, it follows that $u=\,y$. Further, since convex sets in Banach
spaces are weakly closed if and only if they are strongly closed, it
follows that $\,y \in \,K\,.$ So, $K$ is $Y$-closed. Further, from
the uniqueness of the limit $\,y\,,$ we deduce that the whole
sequence $\,y_n\,$
converges weakly in $\,X\,$ to $\,y\,.$ \par%
Finally, to prove that $\,F(y_n)\rightarrow F(y)$ strongly in $Y$ it
is sufficient to show, by using standard arguments, that any
subsequence $y_k$ contains a subsequence $y_m$ such that $\,F(y_m)
\rightarrow F(y)$ strongly in $Y$. Obviously, $y_k\rightharpoonup\,
y $ weakly in $X$. By assumption \emph{ii)}, there is a subsequence
$y_m$ such that  $\,F(y_m)
\rightarrow\, F(y)\,$ strongly in $\,Y\,$. %
This shows that the map $\,F\,$ is continuous on $\,K\,$ with
respect to the $Y$ topology. So, Schauder's fixed point theorem
guarantees the
existence of, at least, one fixed point $\,y_0 \in K\,$, $F(y_0)=\,y_0\,$.\par%
\end{proof}

\section{Appendix}
Our aim is to prove the estimate $$|\,I\,|:=\left|\,\nabla v\cdot
\left( \pa_j\,\nabla v\right) (\pa_j\,v_i)\, \Delta v_i\,\right|\leq
\,|\nabla v|^2\, |D^2\,v|\,|\Delta v|\,.$$ In the sequel, for the
reader's convenience, we avoid the summation convention.\par%
We recall that
$$(\,D^2v_k)^2\,:=\sum_{j,h=1}
^n\!\!\left|\,\pa_{jh}^2\,v_k\,\right|^2 \quad\mbox{ and }\quad
|\,D^2v|^2:=\sum_{k=1}^N (\,D^2v_k)^2\,=\sum_{k=1}^N \sum_{j,h=1}
^n\!\!\left|\,\pa_{jh}^2\,v_k\,\right|^2.$$ We introduce the
$n$-vector $b$ and $N$-vector $w$, whose components are defined as
follows
$$b_j:=(\pa_j\,v)\cdot \, \Delta v\,, \quad w_k^2:=\sum_{j,h=1}^n\left(\,(\pa_h\,v_k)\,b_j\right)^2\,.$$
 The modulus of vector $b$ satisfies the following estimate:
 $$|\,b\,|^2=\sum_{j=1}^n\, b_j^2\,\leq \sum_{j=1}^n|\,\pa_j\,v|^2|\Delta v|^2=|\Delta v|^2
 \sum_{j=1}^n \sum_{i=1}^N(\,\pa_j\,v_i)^2=|\Delta v|^2|\nabla v|^2.$$
Hence \be\label{wk} w_k^2=\sum_{h=1}^n\,\left(\pa_h\,v_k\right)^2
\sum_{j=1}^n\,b_j^2=|\nabla v_k|^2|\Delta v|^2|\nabla v|^2\,.\ee
Moreover
$$\ba{ll}\vs1 \dy|\,I\,|=\left|\,\sum_{k=1}^N\sum_{j,h=1}^n\left(\pa_h\,v_k\right)
\left(\pa^2_{hj}\,v_k\right)\,b_j\right|\leq
\sum_{k=1}^N\left|\,\sum_{j,h=1}^n\left(\pa^2_{hj}\,v_k\right)\,(\pa_h\,v_k)\,b_j
\right|\\\hfill \dy \leq \sum_{k=1}^N\sqrt{
\sum_{j,h=1}^n\left(\pa^2_{hj}\,v_k\right)^2}\,\sqrt{\sum_{j,h=1}^n
\left(\,(\pa_h\,v_k)\,b_j\right)^2}\,,\ea$$ where, in the last step,
we have used that, for any pair of tensors $A$ and $B$,  there holds
$|A\cdot B|\leq |A|\,|B|$. Hence, by the above notations and
estimate \eqref{wk}, we get
$$\ba{ll}\vs1\dy|\,I\,|&\leq \dy \sum_{k=1}^N|D^2v_k|\,|w_k|\leq
\,|\Delta v|\,|\nabla v|\,\sum_{k=1}^N|D^2v_k|\,|\nabla v_k|
\\
\dy &\leq \dy\,|\Delta v|\,|\nabla
v|\sqrt{\sum_{k=1}^N|D^2v_k|^2}\,\sqrt{\sum_{k=1}^N|\nabla v_k|^2}=
\,|\Delta v|\,|\nabla v|^2\,|D^2v|\,,\ea$$ which is our thesis.
  \vskip 0.5cm
\indent
 {\bf Acknowledgments}\,:
The authors like to thank Professor M. Fuchs and Professor P.
Kaplick\'y for giving some interesting references.

\end{document}